\documentclass[10pt]{article}
\usepackage{amsmath,amssymb,amsthm,graphicx,float,url,enumitem,xcolor,subcaption}
\usepackage{placeins}
\usepackage[hidelinks]{hyperref}

\newcommand{\norm}[1]{\left\lVert #1 \right\rVert}

\newcommand{\abs}[1]{\left|#1\right|}

\newcommand*\myat{{\fontfamily{ptm}\selectfont @}}
\DeclareMathOperator{\curl}{\textnormal{curl}\,}

\renewenvironment{proof}{{\bfseries Proof.}}{\qed} 
\newtheorem{theorem}{Theorem}

\newtheorem{definition}[theorem]{Definition}

\numberwithin{equation}{section}
\numberwithin{theorem}{section}

\begin{document}
\title{Eigenvalue Problems in Inverse Electromagnetic Scattering Theory}
\author{S. Cogar\thanks{Department of Mathematical Sciences, University of Delaware, Newark, DE 19716 ({\color{blue}cogar\myat udel.edu}, {\color{blue}colton\myat udel.edu}).} , D. Colton\footnotemark[1] , and P. Monk\thanks{Corresponding author. Department of Mathematical Sciences, University of Delaware, Newark, DE 19716 ({\color{blue}monk\myat udel.edu}).}}
\maketitle

\begin{abstract}

The inverse electromagnetic scattering problem for anisotropic media in general does not have a unique solution. A possible approach to this problem is through the use of appropriate ``target signatures," i.e. eigenvalues associated with the direct scattering problem that are accessible to measurement from a knowledge of the scattering data. In this paper we shall consider three different sets of eigenvalues that can be used as target signatures: 1) eigenvalues of the electric far field operator, 2) transmission eigenvalues and 3) Stekloff eigenvalues.

\end{abstract}

\section{Introduction}

An important unresolved problem in electromagnetic inverse scattering theory is how to detect flaws or changes in the constitutive parameters in an inhomogeneous anisotropic medium. Such a problem presents itself, for example, in efforts to detect structural changes in airplane canopies due to prolonged exposure to ultraviolet radiation and is currently resolved by simply discarding canopies every few months and replacing them with new ones. The difficulties in using electromagnetic waves to interrogate anisotropic media is due to the fact that the corresponding inverse scattering problem no longer has a unique solution even if multiples frequencies and multiple sources are used \cite{gylys-colwell}. Hence alternate approaches to the nondestructive testing of anisotropic materials need to be developed. \par

A possible approach to the target identification problem for anisotropic materials is through the use of appropriate ``target signatures," i.e. eigenvalues associated with the direct scattering problem that are accessible to measurement from a knowledge of the scattering data. The earliest attempt to do this was based on the use of so-called ``scattering resonances" corresponding to the complex poles of the scattering operator. Such an approach appeared particularly fruitful since there is a deep and well-developed theory of such resonances that is readily available to the practitioner \cite{zworski}. However, the use of scattering resonances as target signatures ultimately proved unsuccessful in electromagnetic interrogation due to the difficulty in determining the location of the complex resonances from measured scattering data which is known only for real values of the wave number. \par

A more recent effort to determine appropriate target signatures for anisotropic materials is based on the use of transmission eigenvalues \cite{cakoni_colton_haddar,cakoni_colton_monk}. As opposed to scattering resonances, for dielectrics these eigenvalues are real and can be readily determined from the scattering data. In view of their potential in the nondestructive testing of dielectric materials, we will present the basic theory of transmission eigenvalues in the next two sections of our paper and refer the reader to the two monographs \cite{cakoni_colton_haddar} and \cite{cakoni_colton_monk} for further details. In contrast to the theory of scattering resonances, the theory of transmission eigenvalues is of more recent origin with many questions unanswered. In particular, it has been shown in special cases that complex transmission eigenvalues exist for dielectric materials but whether such eigenvalues exist in general and what their physical meaning is remains an open question. \par

There are two main problems with using transmission eigenvalues as target signatures. The first of these is that such an approach is only applicable to dielectric materials. The second is that one must interrogate the material over a range of frequencies centered at a transmission eigenvalue, i.e. one is forced to use multi-frequency data over a predetermined range of frequencies. A method to overcome both of these difficulties has recently been proposed that is based on using a modified far field operator instead of the standard far field operator that is used to determine both scattering resonances and transmission eigenvalues. In this new approach, the frequency is held fixed and a new artificial eigenparameter is introduced which can be determined from measured scattering data. In one version of this approach the new artificial eigenparameter turns out to be an electromagnetic version of the classical Stekloff eigenvalue problem for elliptic equations and we will discuss this specific class of target signatures in Section 3 of this paper \cite{stek1,stek2}.

\section{Transmission Eigenvalues}

We begin by formulating the direct electromagnetic scattering problem that we will refer to throughout this paper. Let $E^i, H^i$ be an incident field that is scattered by an inhomogeneous object occupying the domain $D$, where we assume that $D$ has smooth boundary $\partial D$. The corresponding scattered field is denoted by $E^s, H^s$ and $E = E^i + E^s, H = H^i + H^s$ is the total field. Then the (normalized) Maxwell's equations are
\begin{equation} \arraycolsep=1.4pt\def\arraystretch{1.5} \begin{array}{r} \curl E - ikH = 0 \\ \curl H + ikN(x)E = 0 \end{array} \quad\text{ in } \mathbb{R}^3 \label{maxwell} \end{equation}
where $k>0$ is the wave number, $x\in\mathbb{R}^3$, $N(x)$ is the symmetric matrix index of refraction with entries in $C^1(\overline{D})$ and $E^s, H^s$ satisfy the Silver-M{\"u}ller radiation condition
\begin{equation} \lim_{r\to\infty} (H^s\times x - rE^s) = 0 \label{smrc} \end{equation}
where $r = \abs{x}$. We will assume that the incident field $E^i, H^i$ is given by
\begin{equation} \arraycolsep=1.4pt\def\arraystretch{1.5} \begin{array}{rclcl} E^i(x) &=& E^i(x;d,p) &=& \frac{i}{k}\curl\curl pe^{ikx\cdot d} \\
								   H^i(x) &=& H^i(x;d,p) &=& \curl pe^{ikx\cdot d} \end{array} \label{incident} \end{equation}
where $d\in\mathbb{R}^3$, $\abs{d}=1$, is the direction of the incident wave and $p\in\mathbb{R}^3$ is the polarization. Under the assumption that
\begin{equation} \arraycolsep=1.4pt\def\arraystretch{1.5} \begin{array}{l} \overline{\xi}\cdot\textnormal{Re} N(x)\xi \ge \alpha\abs{\xi}^2 \\
                                    \overline{\xi}\cdot\textnormal{Im} N(x)\xi \ge 0 \end{array} \label{N_conditions} \end{equation}
for $x\in D$, $\xi\in\mathbb{C}^3$ and some constant $\alpha>0$ it can be shown that there exists a unique solution $E,H\in H_{loc}(\curl,\mathbb{R}^3)$ of \eqref{maxwell}--\eqref{incident} \cite{monk}. \par

From \eqref{maxwell}--\eqref{incident} it is easy to show \cite{colton_kress} that the scattered electric field $E^s(x) = E^s(x;d,p)$ has the asymptotic behavior
\begin{equation} E^s(x;d,p) = \frac{e^{ik\abs{x}}}{\abs{x}}\left\{ E_\infty(\hat{x};d,p) + O\left(\frac{1}{\abs{x}}\right)\right\} \label{Es_asymptotics} \end{equation}
as $\abs{x}\to\infty$ where $\hat{x} = \frac{x}{\abs{x}}$ and $E_\infty$ is the \emph{far field pattern} of the scattered wave. If we define
\begin{equation*} L_t^2(\mathbb{S}^2) := \{g \colon \mathbb{S}^2\to\mathbb{C}^3 : g\in L^2(\mathbb{S}^2), g\cdot\nu = 0\}, \end{equation*}
where $\mathbb{S}^2$ is the unit sphere with unit outward normal $\nu$, the \emph{electric far field operator} $F_e \colon L_t^2(\mathbb{S}^2)\to L_t^2(\mathbb{S}^2)$ is given by
\begin{equation} (F_e g)(\hat{x}) := \int_{\mathbb{S}^2} E_\infty(\hat{x};d,g(d))\, ds(d). \label{effo} \end{equation}
It can easily be seen that $F_e$ is compact \cite{colton_kress}. \par

Of central importance to the inverse scattering problem is the characterization of the null space of the electric far field operator. To this end we define an \emph{electromagnetic Herglotz pair} $(E,H)$ to be a solution of Maxwell's equations
\begin{equation} \arraycolsep=1.4pt\def\arraystretch{1.5} \begin{array}{r} \curl E - ikH = 0 \\ \curl H + ikE = 0 \end{array} \label{maxwell_1} \end{equation}
of the form
\begin{equation} \arraycolsep=1.4pt\def\arraystretch{2} \begin{array}{l} E(x) := \displaystyle\int_{\mathbb{S}^2} E^i(x;d,g(d))\, ds(d) \\
                                  H(x) := \displaystyle\int_{\mathbb{S}^2} H^i(x;d,g(d))\, ds(d) \end{array} \label{herglotz} \end{equation}
with kernel $g\in L_t^2(\mathbb{S}^2)$. The proof of the following theorem can be found in \cite{cakoni_colton_monk}.

\begin{theorem} \label{th_effo} The electric far field operator $F_e \colon L_t^2(\mathbb{S}^2)\to L_t^2(\mathbb{S}^2)$ corresponding to the scattering problem \eqref{maxwell}--\eqref{incident} is injective with dense range if and only if there does not exist a nontrivial solution to the \textnormal{transmission eigenvalue problem}
\begin{equation} \begin{array}{c}\left. \begin{array}{c} \curl\curl E - k^2N(x) E = 0 \\
                                                   \curl\curl E_0 - k^2 E_0 = 0 \end{array} \right\}\text{ in } D \\\\
                                                   
                                 \left. \begin{array}{c} \nu\times E = \nu\times E_0 \\ 
                                   \nu\times \curl E = \nu\times \curl E_0 \end{array} \right\}\quad\quad\text{ on } \partial D \end{array} \label{etep} \end{equation}
where $\nu$ is the outward unit normal to $\partial D$ and $E_0 := E_g, H_0 := H_g$ are an electromagnetic Herglotz pair with kernel $ikg$. 
                                                                      
\end{theorem}

Values of $k$ for which there exist nontrivial solutions to \eqref{etep} are called \emph{transmission eigenvalues}. Transmission eigenvalues play an important role in the theory of inverse scattering. In particular, as we shall see, these eigenvalues can be determined from the far field data and give qualitative information on the anisotropic index of refraction. As noted in the introduction, this is of particular importance in the inverse scattering problem for anisotropic media since the anisotropic material parameters are not uniquely determined from the far field data. The mathematical theory of transmission eigenvalues is based on the following two fundamental results due to Cakoni, Gintides, and Haddar \cite{cakoni_gintides_haddar} (see also \cite{cakoni_colton_monk}), where for real $N(x)$ we define
\begin{equation*} n_* := \inf_{x\in D} \inf_{\norm{\xi}=1} \overline{\xi}\cdot N(x)\xi, \quad n^* := \sup_{x\in D} \sup_{\norm{\xi}=1} \overline{\xi}\cdot N(x)\xi. \end{equation*}

\begin{theorem} \label{th_existence} Assume that for every $\xi\in\mathbb{C}^3$, $\abs{\xi} = 1$, and some constants $\alpha>0$, $\beta>0$ one of the following inequalities is valid:
\begin{enumerate}[label=\arabic*)]

\item $1+\alpha\le n_* \le \overline{\xi}\cdot N(x)\xi \le n^* < \infty$, $x\in D$,

\item $0<n_* \le \overline{\xi}\cdot N(x)\xi \le n^* \le 1-\beta$, $x\in D$.

\end{enumerate}

Then there exists an infinite countable set of positive transmission eigenvalues corresponding to \eqref{etep} with $+\infty$ as the only accumulation point.

\end{theorem}

Note that, in contrast to scattering resonances, the above theorem says that for real $N(x)$ there exist positive transmission eigenvalues and, as we shall see in the next section, these can be determined from measured far field data and thus can be used as target signatures. It can be shown (c.f. Theorem 8.12 of \cite{colton_kress}) that if $N(x)$ is not real-valued then positive transmission eigenvalues do not exist.

\begin{theorem} \label{th_monotone} Let $k_{1,D,N(x)}$ be the first positive transmission eigenvalue for \eqref{etep} and let $\alpha$ and $\beta$ be positive constants. Denote by $k_{1,D,n_*}$ and $k_{1,D,n^*}$ the first positive transmission eigenvalue of \eqref{etep} for $N = n_*I$ and $N = n^*I$, respectively, and let $\norm{\cdot}_2$ denote the Euclidean operator norm.
\begin{enumerate}[label=\arabic*)]

\item If $\norm{N(x)}_2 \ge \alpha > 1$, then $0 < k_{1,D,n^*} \le k_{1,D,N(x)} \le k_{1,D,n_*}$.

\item If $0 < \norm{N(x)}_2 \le 1-\beta$, then $0 < k_{1,D,n_*} \le k_{1,D,N(x)} \le k_{1,D,n^*}$.

\end{enumerate}

\end{theorem}

Assuming that $k_{1,D,N(x)}$ can be computed from the far field measurements, Theorem \ref{th_monotone} provides an approach to obtaining qualitative information on $N(x)$ by computing a constant $n$ such that $k_{1,D,N(x)}$ is the first positive transmission eigenvalue corresponding to \eqref{etep} with $N := nI$ for this $n$. The above theorem then implies that $n_*\le n\le n^*$. Since $N(x)$ is positive definite, $n_* = \lambda_1$ and $n^* = \lambda_3$ where $\lambda_1$ is the smallest and $\lambda_3$ is the largest eigenvalue of $N(x)$. As an example, let $D = (0,1)\times(0,1)$ and
\begin{equation*} N = \left(\begin{array}{cc} 1/6 & 0 \\ 0 & 1/8 \end{array}\right). \end{equation*}
Then $\lambda_1 = 0.125$, $\lambda_3 = 0.166$ and the computed $n = 0.135$ \cite{cakoni_colton_monk_sun}.

\subsection{Measurement of Transmission Eigenvalues}

We will now consider the problem of determining transmission eigenvalues from the measured scattering data. In particular, we will assume that the index of refraction is real-valued and make use of Theorems \ref{th_effo} and \ref{th_existence}. To this end, we need the following theorem from \cite{colton_kress1993} (the result in \cite{colton_kress1993} assumed that $N(x)$ was a scalar but the same proof is valid for $N(x)$ a symmetric matrix satisfying the assumption \eqref{N_conditions}).

\begin{theorem} \label{th_effo_eigs} Let $E_g^i, H_g^i$ and $E_h^i, H_h^i$ be electromagnetic Herglotz pairs with kernels $g,h\in L_t^2(\mathbb{S}^2)$, respectively, and let $E_g$ and $E_h$ be the solutions of \eqref{maxwell}--\eqref{incident} with $E^i, H^i$ replaced by $E_g^i, H_g^i$ and $E_h^i, H_h^i$ respectively. Then
\begin{equation} k\iint_D \textnormal{Im} N(x) E_g\cdot\overline{E_h}\, dx = -2\pi(F_eg,h) - 2\pi(g,F_eh) - (F_eg,F_eh) \label{effo_eigs} \end{equation}
where $(\cdot,\cdot)$ denotes the inner product on $L_t^2(\mathbb{S}^2)$.

\end{theorem}

If $\textnormal{Im}N(x) = 0$ then it is an easy consequence of this theorem that the compact operator $F_e$ is normal and hence has an infinite number of eigenvalues \cite{colton_kress1993}. In this case it can also easily be seen from \eqref{effo_eigs} that if $F_eg = \lambda g$ then
\begin{equation*} 0 = -2\pi(\lambda g,g) - 2\pi(g,\lambda g) - \abs{\lambda}^2(g,g) \end{equation*}
which implies that
\begin{equation} \abs{\lambda + 2\pi} = 2\pi \label{effo_eigs_circle} \end{equation}
i.e. the eigenvalues of the electric far field operator all lie on the circle \eqref{effo_eigs_circle}. A similar calculation can be done if, instead of using the electric far field operator, we use the \emph{magnetic far field operator}, i.e. if
\begin{equation} H^s(x;d,p) = \frac{e^{ik\abs{x}}}{\abs{x}}\left\{ H_\infty(\hat{x};d,p) + O\left(\frac{1}{\abs{x}}\right)\right\} \label{Hs_asymptotics} \end{equation}
and the magnetic far field operator $F_m \colon L_t^2(\mathbb{S}^2)\to L_t^2(\mathbb{S}^2)$ is defined by
\begin{equation} (F_mg)(\hat{x}) := \int_{\mathbb{S}^2} H_\infty(\hat{x};d,g(d))\, ds(d). \label{mffo} \end{equation}
It is again easily seen that $F_m$ is compact. In a similar manner to the electric far field operator, it can be shown that $F_m$ is injective with dense range provided $k$ is not an eigenvalue of the interior transmission problem
\begin{equation} \begin{array}{c} \begin{array}{c} \curl(N(x)^{-1}\curl H) - k^2 H = 0 \\
                                                   \curl\curl H_0 - k^2 H_0 = 0 \end{array} \text{ in } D \\\\
                                                   
                                  \begin{array}{c} \nu\times E = \nu\times E_0 \\ 
                                   N(x)^{-1}(\nu\times \curl E) = \nu\times \curl E_0 \end{array} \text{ on } \partial D \end{array} \label{mtep} \end{equation}
and that if $N(x)$ is real then the compact operator $F_m$ is normal and hence has an infinite number of eigenvalues. An identity analogous to \eqref{effo_eigs} can also be established for the magnetic far field operator $F_m$ \cite{kirsch_grinberg} and used to show that the eigenvalues of $F_m$ all lie on the circle
\begin{equation} \abs{\lambda - \frac{2\pi i}{k}} = \frac{2\pi}{k} \label{mffo_eigs_circle}. \end{equation}
If $N(x)$ is not real, then we may still establish the existence of infinitely many eigenvalues of $F_e$ and $F_m$ using Lidski's theorem \cite{colton_kress}, as we show in the following theorem. We first remark that both $F_e$ and $F_m$ are trace-class operators, as can be seen by considering truncated spherical harmonic expansions of the kernel of each operator.

\begin{theorem}

\label{th_lidski}

If $\textnormal{Im} N(x)$ is positive on a nonempty open set in $D$, then $F_e$ has infinitely many eigenvalues.

\end{theorem}

\begin{proof} Since $F_e$ is a trace-class operator, by Lidski's theorem it remains to show that $F_e$ has a finite-dimensional nullspace and an imaginary part which is nonnegative. Unfortunately, the formula \eqref{effo_eigs} does not provide the second requirement, and we instead show it for a slightly modified operator $\tilde{F}_e$. In order to prove the first part, we show that under our assumption on $N$ no real transmission eigenvalues can exist, from which Theorem \ref{effo} implies that $F_e$ is injective. Indeed, if $E, E_0$ satisfies the homogeneous interior transmission problem \eqref{etep}, then we see from the equation for $E_0$ in $D$ and the integration by parts formula for the curl operator that
\begin{equation*} \int_{\partial D} \left[ \left( \curl E_0\times\overline{E_0} \right)\cdot\nu - \left(\curl\overline{E_0}\times E_0\right)\cdot\nu \right] ds = 0. \end{equation*}
Applying the vector identity $(\mathbf{a}\times\mathbf{b})\cdot\mathbf{c} = -\mathbf{a}\cdot(\mathbf{c}\times\mathbf{b})$, the boundary conditions, and the same vector identity again yields
\begin{equation*} \int_{\partial D} \left[ \left( \curl E\times\overline{E} \right)\cdot\nu - \left(\curl\overline{E}\times E\right)\cdot\nu \right] ds = 0, \end{equation*}
and it follows from another application of the integration by parts formula that
\begin{align*} 0 &= \iint_D \left( \curl\curl E\cdot\overline{E} - E\cdot\curl\curl\overline{E} \right) dx
\\&= 2ik^2 \iint_D \text{Im} N(x) \abs{E}^2 dx. \end{align*}
Thus, we observe that $E=0$ on the open set $D_0 := \{x\in D\,:\, \text{Im} N(x) = 0\}$, and by the unique continuation principle it follows that $E=0$ in all of $D$. This result implies that $E_0=0$ as well, and we conclude that $k$ is not a transmission eigenvalue. 

In order to prove the second part, we rewrite \eqref{effo_eigs} in terms of $\tilde{F}_e := -ik F_e$ as
\begin{equation} ik^2\iint_D \textnormal{Im} N(x) E_g\cdot\overline{E_h}\, dx = 2\pi(\tilde{F}_eg,h) - 2\pi(g,\tilde{F}_eh) - \frac{i}{k}(\tilde{F}_eg,\tilde{F}_eh) \label{effo_eigs2}, \end{equation}
from which it follows that for all $g\in L_t^2(\mathbb{S}^2)$ we have
\begin{align*} \text{Im}(\tilde{F}_eg,g) &= \frac{1}{2i} \left[ (\tilde{F}_eg,g) - (g,\tilde{F}_eg) \right]
\\&= \frac{1}{4\pi i} \left[ ik^2\iint_D \text{Im}N(x)\abs{E_g}^2 dx + \frac{i}{k} \norm{\tilde{F}_eg}^2 \right]
\\&\ge0. \end{align*}
Therefore, the assumptions of Lidski's theorem are satisfied for the operator $\tilde{F}_e := -ikF_e$, and we conclude that $\tilde{F}_e$ and hence $F_e$ has infinitely many eigenvalues. \end{proof}

\vspace{.2in}

A similar computation establishes the result for the magnetic far field operator $F_m$. Note that the definition of the electric and magnetic far field operators in \cite{kirsch_grinberg} differ by a factor of $4\pi$ from the ones that we are using. \par

We now present two methods for determining transmission eigenvalues from the measured scattering data. We first note that the transmission eigenvalue problems \eqref{etep} and \eqref{mtep} are seen to be equivalent by a simple change of dependent variables and hence have the same eigenvalues. Hence there is no ambiguity in simply referring to the eigenvalues of \eqref{etep} and \eqref{mtep} as transmission eigenvalues. We will restrict our attention to considering $H_\infty(\hat{x};d,p)$. We always assume that $\textnormal{Im} N = 0$ and that $D$ is known ($D$ can be determined by using the linear sampling method --- c.f. \cite{cakoni_colton_monk}). \par

We first show how transmission eigenvalues can be determined from the magnetic far field operator $F_m$.

\begin{definition} \textnormal{If the solution $E_0$ of \eqref{etep} is the electric field of an electromagnetic Herglotz pair then we call the transmission eigenvalue $k$ a} nonscattering wave number.

\end{definition}

It is clear that the concept of nonscattering wave numbers is far more restrictive than the concept of transmission eigenvalues. Indeed, the only case known to date when a transmission eigenvalue is a nonscattering wave number is the case when $D$ is a ball and $N(x) = n(\abs{x})I$. We define
\begin{equation} H_{e,\infty}(\hat{x};z,p) := \frac{ik}{4\pi}(\hat{x}\times p)e^{-ik\hat{x}\cdot z} \label{ffp_dipole} \end{equation}
where $z\in\mathbb{R}^3$ and note that the right-hand side of \eqref{ffp_dipole} is the far field pattern of the magnetic field of an electric dipole. We now let $g_z^\alpha\in L_t^2(\mathbb{S}^2)$ be the Tikhonov regularized solution of the magnetic far field equation
\begin{equation} (F_mg)(\hat{x}) = H_{e,\infty}(\hat{x};z,p) \label{mffe} \end{equation}
i.e. $g_z^\alpha$ is the solution to
\begin{equation} (\alpha I + F_m^*F_m)g_z^\alpha = F_m^* H_{e,\infty} \label{mffe_reg}. \end{equation}
We then have the following result (c.f. Theorem 4.44 of \cite{cakoni_colton_haddar} for the scalar case; the proof in the vector case proceeds in the same manner).

\begin{theorem} \label{th_determine} Assume that $D$ is simply connected and that $N(x)$ satisfies one of the two conditions stated in Theorem \ref{th_existence}. Assume further that $k$ is not a nonscattering wave number and let $Hg_z^\alpha$ denote the magnetic field of the electromagnetic field defined by \eqref{herglotz}. Then for any ball $B\subset D$, $\norm{Hg_z^\alpha}_{L^2(D)}$ is bounded as $\alpha\to0$ for almost every $z\in B$ if and only if $k$ is not a transmission eigenvalue.

\end{theorem}

In particular, if one plots $k$ versus $\norm{g_z^\alpha}_{L^2(\mathbb{S}^2)}$ for several choices of points $z$ then the location of transmission eigenvalues will appear as sharp peaks in the graph (for the scalar case see Figure 4.2 of \cite{cakoni_colton_haddar}). \par

We now turn our attention to a second method for determining transmission eigenvalues from the magnetic far field operator $F_m$ which is based on the behavior of the phase of the eigenvalues of the compact normal operator $F_m$. To this end, we recall that if $k>0$ is not a transmission eigenvalue then $F_m = F_{m,k}$ is injective where we have explicitly noted the dependence of $F_m$ on $k$. Hence if $k>0$ is not a transmission eigenvalue, we have the existence of a complete orthonormal basis $(g_j(k))_{j=1}^\infty$ of $L^2(\mathbb{S}^2)$ such that
\begin{equation} F_{m,k} g_j(k) = \lambda_j(k) g_j(k) \label{mffo_basis} \end{equation}
where $\lambda_j(k)\neq0$ forms a sequence of complex numbers that goes to zero as $j\to\infty$. Define
\begin{equation} \hat{\lambda}_j(k) := \frac{\lambda_j(k)}{\abs{\lambda_j(k)}} \label{mffo_eigs_hat}. \end{equation}
We then have the following theorem due to Lechleiter and Rennoch \cite{lechleiter_rennoch}:

\begin{theorem} \label{th_inside_outside} Assume that condition 1 (respectively, condition 2) of Theorem \ref{th_existence} is valid. Let $k_0>0$ and let $(k_\ell)$ be a sequence of positive numbers converging to $k_0$ as $\ell\to\infty$. Assume there exists a sequence $(\hat{\lambda}_\ell) = \hat{\lambda}_{j_\ell}(k_\ell)$ for some index $j_\ell$ such that $\hat{\lambda}_{\ell}\to -1$ (respectively $\hat{\lambda}_\ell\to +1$) as $\ell\to\infty$. Then $k_0$ is a transmission eigenvalue.

\end{theorem}

Note that since $F_{m,k}$ is compact and all the eigenvalues lie on the circle \eqref{mffo_eigs_circle}, the only possible accumulation points of the sequence $\hat{\lambda}_\ell$ are $-1$ and $+1$. \par

The criterion of Theorem \ref{th_inside_outside} can be used as an indicator of transmission eigenvalues. However, the hard part is to prove that it occurs for every transmission eigenvalue. We refer the reader to \cite{lechleiter_rennoch} for a further discussion on this issue.

\section{Stekloff Eigenvalues}
So far we have seen two families of eigenvalues that can be determined from scattering data:
\begin{description}
\item[{\bf{}Eigenvalues of the Electric Far Field Operator}:]  These can be computed directly from the far field pattern 
using single frequency data.  However, it is not easy to determine how changes in the material properties of the object
(i.e. $N(x)$) perturb the eigenvalues.
\item[{\bf{}Transmission Eigenvalues:}]  These have a direct relation to $N(x)$ as shown in Theorem~\ref{th_monotone}. However, they have to be computed using multi-frequency data and can only be determined for dielectric scatterers.  
\end{description}
We shall now introduce a family of eigenvalues from \cite{stek2} that can be computed from the far field pattern at a single frequency, and for which a simple perturbation theory is known.  This is achieved by constructing a modified far field operator using an auxiliary problem which includes an appropriate eigenparameter. 

To define this problem choose a domain $B$ such that either 1) $B=D$ or 2) $B$ is a ball containing $\overline{D}$ in its interior.  We also need an operator $S:L_t^2(\partial B)\to L_t^2(\partial B)$ such that $S$ is self-adjoint, bounded and
\[
\langle Su,u\rangle\geq 0\mbox{ for all } u\in L_t^2(\partial B),
\]
where $\langle\cdot,\cdot\rangle$ is the $L^2$ inner product on $\partial B$. Next we define, for any sufficiently smooth vector field $w$, the tangential component of $w$ on $\partial B$ by
\[
w_T=(\nu\times w)\times \nu\mbox{ on  }\partial B.
\]
Finally we need to choose an impedance parameter $\lambda\in \mathbb{R}$ with $\lambda>0$ (note that a standard impedance parameter
would typically be complex). 
Now we can define the solution $E_S$ of the following generalized impedance problem 
\begin{eqnarray} 
\curl\curl E_S- k^2 E_S&=& 0  \text{ in } \mathbb{R}^3\setminus \overline{B}\\
\nu\times \curl E_S &=& \lambda SE_{S,T}  \text{ on } \partial B\\
E^i+E^s_S&=&E_S\mbox{ in } \mathbb{R}^3\setminus B\\
\lim_{r\to\infty}\left(\curl E_S^s\times x-ikrE_S^s\right)&=&0.
 \label{Sscat} \end{eqnarray}
 This scattering problem has a unique solution for any $k>0$ as shown in \cite{stek2} (for any  $\lambda$, any solution is always unique).
 
 Then, as usual for a scattering problem, the scattered field $E_S^s$ has the asymptotic expansion
 \[
 E_S^s(x)=\frac{\exp(ikr)}{r}E_{S,\infty}(\hat{x},d;p)+O\left(\frac{1}{r^2}\right)\mbox{ as }r\to\infty,
 \]
 and we can then define the  impedance far field operator by
 \[
 (F_Sg)(\hat{x})=\int_{\mathbb{S}^2}E_{S,\infty}(\hat{x};d,g(d))\,ds(d).
 \]
 The modified far field operator is then defined by
 \[
 F_M=F_m-F_S.
 \]
We can see a link between the modified far field operator and the interior Stekloff eigenvalue problem as 
argued in \cite{stek2}.  There it is shown that $F_M$ is injective with dense range provided $\lambda$ is not a generalized Stekloff eigenvalue of the problem
\begin{eqnarray}
\curl\curl w-k^2N w&=&0\mbox{ in }B,\\
\nu\times\curl w-\lambda S w_T&=&0\mbox{ on }\partial B.
\end{eqnarray}
It is then necessary to analyze the existence of generalized Stekloff eigenvalues, and this analysis depends on the choice of $S$.  The most obvious choice corresponding to the standard impedance boundary condition is $S=I$.  Unfortunately, direct calculation of the eigenvalues in the case when $N=1$ and $B$ is a ball shows that there are two families of eigenvalues having different accumulation points (one at infinity and one at zero).  Indeed in this case, assuming $N=1$, if $\lambda$ is an eigenvalue then so is $-k^2/\lambda$. Thus they cannot be analyzed as the 
eigenvalues of a compact operator.

Instead, in \cite{stek2} we make the choice of $S$ as follows.  Let $u\in L^2_t(\partial B)$ and define 
$q\in H^1(\partial B)/\mathbb{R}$ by solving
\[
\Delta_{\partial B}q=\curl_{\partial B} u.
\]
Note that this assumes that if $B=D$ then $\partial D$ has just one connected component.
Then $Su=\vec{\curl}_{\partial B}u$.  Here $\Delta_{\partial B}$ is the Laplace-Beltrami operator on $\partial B$,
and $\curl_{\partial B}$ and $\vec{\curl}_{\partial B}$ are the scalar and vector surface curls, respectively.

Using $S$ we can now write the generalized Stekloff eigenvalue problem as an operator equation.
We introduce the operator $T:H({\rm{}div}_{\partial B}^0,\partial B)\to H({\rm{}div}^0_{\partial B},\partial B)$ where
\[
H({\rm{}div}^0_{\partial B},\partial B)=\{ u\in L_t^2(\partial B)\;|\; \nabla_{\partial B}\cdot u=0 \mbox{ on }\partial B\}
\]
defined as follows.  For given $f\in H({\rm{}div}^0_{\partial B},\partial B)$ we define $w$ to be the weak solution
of
\begin{eqnarray*}
\curl\curl w-k^2\epsilon_rw&=&0\mbox{ in } B\\
\nu\times\curl w&=&-f\mbox{ on }\partial B.
\end{eqnarray*}
This is a well posed problem provided $k^2$ is not an interior Neumann eigenvalue for the curl-curl operator. These eigenvalues form a discrete set and from now on we assume $k^2>0$ is not such an eigenvalue.
Then 
\[
Tf=Sw_T\mbox{ on }\partial B.
\]
The fact that $Sw_T$ is surface divergence free can be used to show that $Tf$ is actually in $(H^{1/2}(\partial B))^3$ and hence the operator $T$ is compact.  Furthermore it is self adjoint, and consequently there exist infinitely many eigenvalues
$\mu$ with associated eigenfunction $u\not=0$ for the problem
\[
Tu=\mu u.
\]
Considering the definition of $T$, we see that if $\mu$ is an eigenvalue for $T$ then $\lambda=-1/\mu$ is a generalized Stekloff eigenvalue.  Thus we conclude:
\begin{theorem}[Thm. 3.6 of \cite{stek2}] When $\epsilon_r$ is real, and $k^2$ is not an interior Neumann eigenvalue for 
the curl-curl operator, there exists a countable set of real generalized Stekloff eigenvalues that accumulate at infinity.
\end{theorem}

Supposing now that we can measure generalized Stekloff eigenvalues, we can assume that changes in these eigenvalues can give information about changes in $N(x)$ as is the case for the Helmholtz equation~\cite{stek1}.
To see this, suppose $(w,\lambda)$, $w\not=0$ is a generalized Stekloff eigenpair for permittivity $N(x)$ and that 
$(w_\delta,\lambda_\delta)$ is the corresponding eigenpair for $N(x)+\delta N(x)$ where $\Vert\delta N\Vert_{L^\infty}$ is small.
Then assuming that $w\approx w_\delta$ (for example, when the eigenvalue is simple and the perturbation
$\delta N$ is small), we have, neglecting quadratic terms, that
\begin{equation} \lambda-\lambda_\delta\approx -k^2\frac{(\delta Nw,w)}{\langle Sw_T,Sw_T\rangle} \label{shift_estimate} \end{equation}
where $\langle .,.\rangle$ is the $L^2$ inner product on $\partial B$ and $(.,.)$ is the $L^2$ inner product on $B$.

The main question now is how to determine generalized Stekloff eigenvalues (or at least a few of them) from far field scattering data.  As in the case of transmission eigenvalues, this involves the far field equation, and this time we use the
electric far field equation.  The outgoing electric field due to a point dipole at position $z$ with polarization $q$ in free space is~\cite{colton_kress}
\[
E_e = \frac{i}{k}\curl_x\curl_x(q\Phi(x,z))
\]
where $\curl_x$ denotes the curl with respect to $x$ and $\Phi$ is the fundamental solution of the Helmholtz equation
\[
\Phi(x,y)=\frac{\exp(ik|x-y|)}{4\pi |x-y|}, \; x\not= y.
\]
The far field pattern due to the dipole source is then given by
\[
E_{e,\infty}(\hat{x},z;q)=\frac{ik}{4\pi}(\hat{x}\times q)\times\hat{x} \exp(-ik\hat{x}\cdot z)
\]
where $\hat{x}\in \mathbb{S}^2$ is the observation direction (for comparison, see (\ref{Hs_asymptotics}) for the definition of the 
magnetic far field pattern, and (\ref{ffp_dipole}) for the magnetic far field pattern of a dipole source).

For generalized Stekloff eigenvalues, the far field equation corresponding to (\ref{mffe}) is then to seek $g_{z,q}\in L^2_t(\mathbb{S}^2)$ such that
\begin{equation}
(F_Mg_{z,q})(\hat{x})=E_{e,\infty}(\hat{x},z;q) \mbox{ for all }\hat{x}\in\mathbb{S}^2.\label{FMeq}
\end{equation}
As in the case of transmission eigenvalues, we actually solve a Tikhonov regularized version of this problem by 
choosing a regularization parameter $\alpha>0$ and solving
\[
(\alpha I+F_M^*F_M)g_{z,q,\alpha}=F_M^*E_{{e,\infty}}
\]
where $F_M^*$ is the $L^2$ adjoint of $F_M$.  Note that $F_M$ depends on the Stekloff parameter $\lambda$, so
$g_{z,q,\alpha}$ is also dependent on $\lambda$.  As $\lambda$ varies, we can use $\Vert g_{z,q,\alpha}\Vert_{L^2_t(\mathbb{S}^2)}$ as an indicator function for Stekloff eigenvalues.  Although we cannot prove that this is an appropriate indicator function, we
can prove that there is an approximate solution of (\ref{FMeq}) that does have this property.  All
numerical tests suggest that the solution of the above regularized problem can indeed serve as an indicator function.

In a similar way to the proof of Theorem~\ref{th_determine} we can now prove the analogous result for Stekloff eigenvalues. To do this we need to recall the definition of the electric Herglotz wave function
\[
v_g(x)=-ik\int_{\mathbb{S}^2}g(d)\exp(-ikx\cdot d)\,ds_d.
\]

\begin{theorem}[Thm. 4.2 of \cite{stek2}] Assume $\lambda$ is not a Stekloff eigenvalue and $k^2$ is not an interior 
Neumann eigenvalue for the curl-curl problem. Let $z\in D$ and $q$ be fixed.  Then for every $\epsilon>0$ there exists a function $g_\epsilon\in L^2_t(\mathbb{S}^2)$ that satisfies
\[
\lim_{\epsilon\to 0}\Vert F_Mg_\epsilon-E_{e,\infty}(\cdot,z;q)\Vert_{L^2_t(\mathbb{S}^2)}=0
\]
and such that $\Vert v_{g_{\epsilon}}\Vert_{L^2_t(B)}$ is bounded as $\epsilon\to 0$.
\end{theorem}

Conversely we can show that if $\lambda$ is a Stekloff eigenvalue, $\Vert v_{g_{\epsilon}}\Vert_{L^2_t(B)}$
cannot remain bounded as $\epsilon\to 0$ for almost every $z\in D$.
These results suggest that a graph of $\Vert v_{g_{\epsilon}}\Vert_{L^2_t(\mathbb{S}^2)}$  against $\lambda$ will
show peaks at the Stekloff eigenvalues (provided we sample several points $z\in D$).
In practice we do not use $\Vert v_{g_{\epsilon}}\Vert_{L^2_t(\mathbb{S}^2)}$ to detect eigenvalues because it is 
somewhat expensive to compute (we would replace $g_\epsilon$ by $g_{z,q,\alpha}$).  Instead we use as a surrogate $\Vert g_{z,q,\alpha}\Vert_{L^2_t(\mathbb{S}^2)}$.
\section{Numerical Examples}
Numerous examples of the computation of transmission eigenvalues exist in the literature (c.f.~\cite{cakoni_colton_monk}) and so we will not present more here.  Instead we will focus on the two sets of eigenvalues discussed in this paper that
can be computed at a single frequency: 1) eigenvalues of the electric far field operator and 2) generalized Stekloff eigenvalues.

Our numerical examples are all computed using synthetic far field data.  This data is computed using the Netgen~\cite{netgen} finite element library using second order edge elements and a fifth order approximation to curved surfaces.  We use a spherical Perfectly Matched Layer, at a distance of half a wavelength from the circumscribing sphere for $B$, of thickness one quarter of a wavelength.  The PML parameter is chosen to give approximately 0.6\% relative error in the computed far field pattern for scattering by a penetrable sphere of unit radius
(measured in the $L^2$ norm).  In all the calculations the wave number is chosen to be $k=1$ so the wavelength in free space is $2\pi$. 

The far field pattern $F_S$ of the generalized Stekloff scattering problem needed for the solution of (\ref{FMeq}) is computed by the same code with the addition of the calculation of an approximation to the operator $S$ computed using third order finite elements in $H^1(\partial B)$.  Generalized Stekloff eigenvalues for arbitrary structures are computed using the same finite elements but now on
a bounded domain as described in \cite{stek2}.

The far field operators are discretized by quadrature on the unit sphere.  We use a finite element grid on the unit sphere having 99 nodes (made by netgen) and use vertex based quadrature on each element to calculate the weights for each
vertex value of the far field pattern.

Two domains are considered for the scatterer.  The first is the unit cube, and the second is the (hockey) puck which is a circular cylinder of radius 3/2 and unit height centered at the origin.  The latter scatterer has been suggested as a good
experimental model, being dielectric and which can easily be damaged by drilling out portions.  Experimental results are
not considered here.

\subsection{Eigenvalues of the Far Field Operator}

In this section we investigate the use of eigenvalues of the electric far field operator as a target signature. Due to the ease of computing such eigenvalues, they seem to be a natural choice for this purpose, but a significant drawback is the lack of theory concerning their response to changes in the material parameters of an inhomogeneous medium. Thus, our study is confined to a collection of numerical examples, and to facilitate a direct comparison we perform the same numerical experiments as we will for Stekloff eigenvalues. In order to compute the eigenvalues of the electric far field operator $F_e$, we first discretize the operator using quadrature to obtain a matrix $A$. When we investigate the effect of noisy data, we obtain a noisy far field matrix $A^\varepsilon$ by multiplying each component of the far field data by $1+\varepsilon\frac{\zeta + i\mu}{\sqrt{2}}$, where $\varepsilon>0$ is a fixed parameter and $\zeta,\mu$ are both uniformly distributed random numbers in $[-1,1]$ computed using the \texttt{rand} command in \texttt{MATLAB}. The eigenvalues of $A^\varepsilon$ are then computed using the \texttt{eig} command in \texttt{MATLAB}. In Figure \ref{FFeigs} we see that the eigenvalues of the far field operator for both the unit cube and the puck lie on the circle $\abs{\lambda+2\pi} = 2\pi$ as implied by Theorem \ref{th_effo_eigs}.

\begin{figure}
\begin{subfigure}{.5\textwidth}
\centering
\includegraphics[width=1\linewidth]{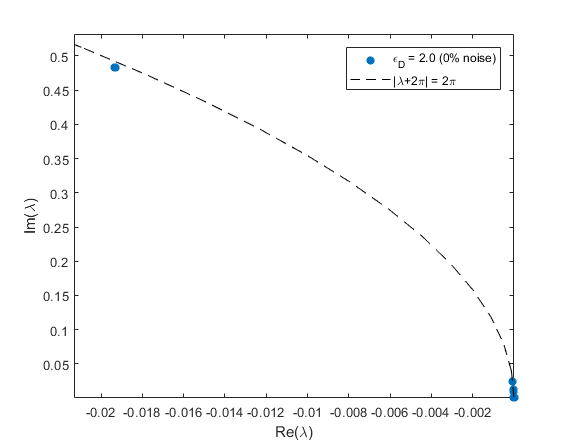}
\caption{unit cube}
\label{FFeigs_CS}
\end{subfigure}
\begin{subfigure}{.5\textwidth}
\centering
\includegraphics[width=1\linewidth]{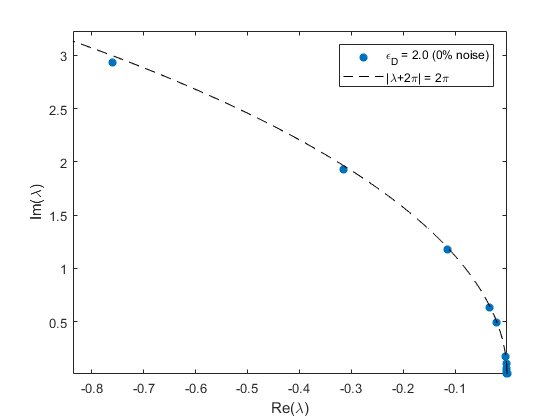}
\caption{puck}
\label{FFeigs_puckS}
\end{subfigure}
\caption{The computed eigenvalues of the electric far field operator with $\epsilon_D = 2$ and no noise. The eigenvalues lie on the circle $\abs{\lambda+2\pi}=2\pi$ and appear to converge to zero as predicted.}
\label{FFeigs}
\end{figure}

An important property of a target signature is that it is stable in the presence of noise. In Figure \ref{FFeigs_noise} we plot the eigenvalues of the far field operator for both the unit cube and puck with $\epsilon_D = 2$ for different amounts of noise, and in Figure \ref{FFeigs_complex_noise} we perform the same test with $\epsilon_D = 2+2i$. In the presence of absorption (complex $\epsilon_D$) the eigenvalues move inside the  circle $\abs{\lambda+2\pi} = 2\pi$.

We remark that although the eigenvalues near the origin are highly sensitive to noise, the eigenvalues with larger magnitude tend to remain localized. This stability is promising, and the distribution of the eigenvalues near the origin may even provide some measure of the noise level.

\begin{figure}
\begin{subfigure}{.5\textwidth}
\includegraphics[width=1\linewidth]{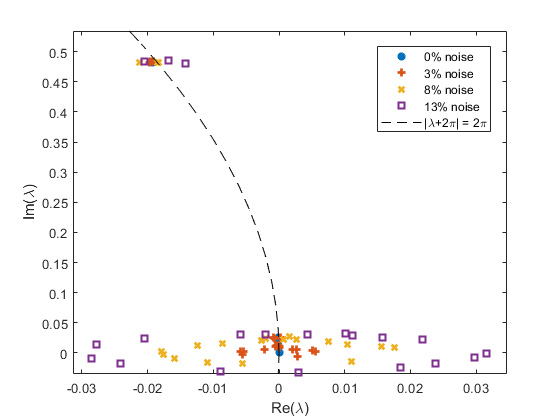}
\caption{unit cube}
\end{subfigure}
\begin{subfigure}{.5\textwidth}
\includegraphics[width=1\linewidth]{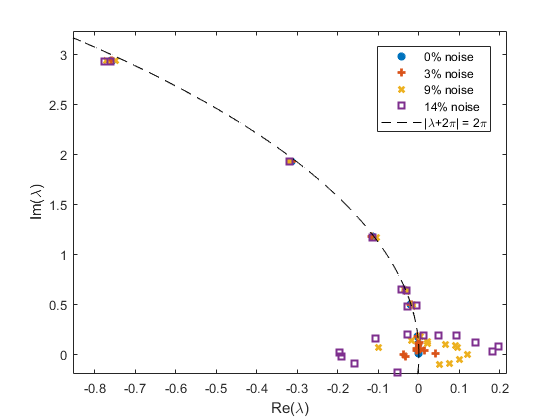}
\caption{puck}
\end{subfigure}
\caption{The computed eigenvalues of the electric far field operator with $\epsilon_D = 2$ and various levels of noise. The eigenvalues of larger magnitude remain stable in the presence of noise, whereas those near the origin are highly unstable.}
\label{FFeigs_noise}
\end{figure}

\begin{figure}
\begin{subfigure}{.5\textwidth}
\includegraphics[width=1\linewidth]{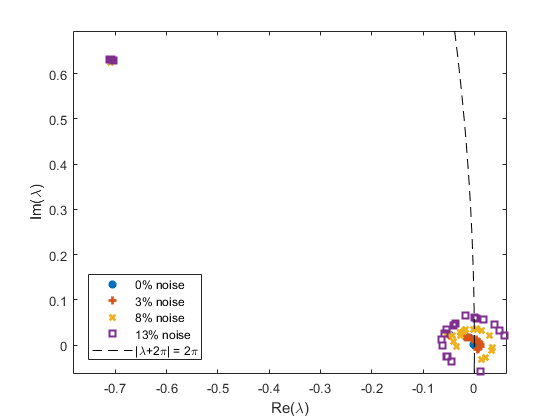}
\caption{unit cube}
\end{subfigure}
\begin{subfigure}{.5\textwidth}
\includegraphics[width=1\linewidth]{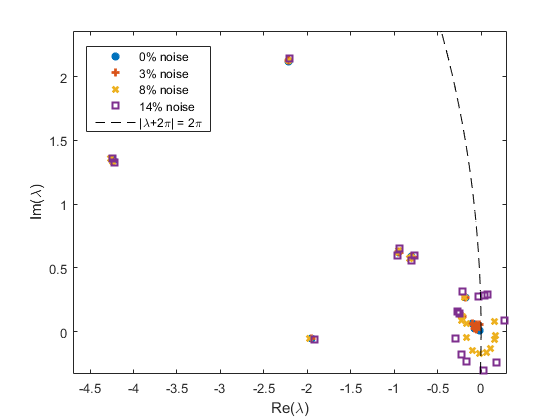}
\caption{puck}
\end{subfigure}
\caption{The computed eigenvalues of the electric far field operator with $\epsilon_D = 2+2i$ and various levels of noise. The eigenvalues of larger magnitude remain stable in the presence of noise, whereas those near the origin are highly unstable.}
\label{FFeigs_complex_noise}
\end{figure}

Of course, our primary point of inquiry is whether the eigenvalues of the far field operator reliably shift due to a change in an inhomogeneous medium. In Figure \ref{FFeigs_shift} we plot the eigenvalues corresponding to $\epsilon_D = 2$ and $\epsilon_D = 2.5$ for both the unit cube and puck. We remark that the eigenvalues with larger magnitude exhibit a noticeable shift due to this change, which are precisely the eigenvalues that remained stable in the presence of noise in our previous test.

\begin{figure}
\begin{subfigure}{.5\textwidth}
\includegraphics[width=1\linewidth]{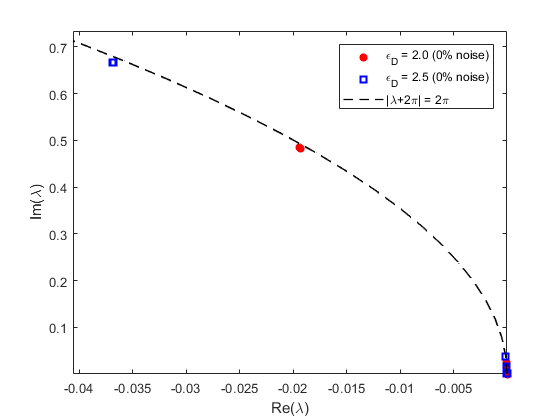}
\caption{unit cube}
\end{subfigure}
\begin{subfigure}{.5\textwidth}
\includegraphics[width=1\linewidth]{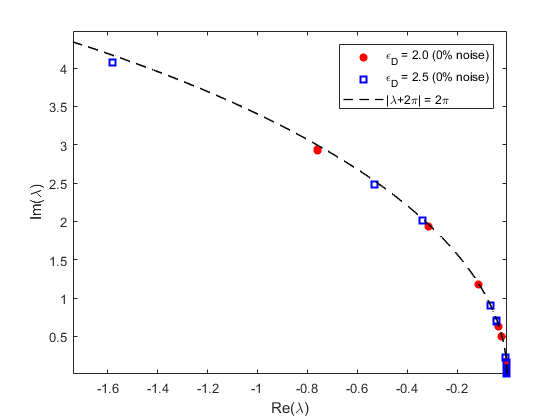}
\caption{puck}
\end{subfigure}
\caption{The computed eigenvalues of the electric far field operator with $\epsilon_D = 2$ and $\epsilon_D = 2.5$, where no noise has been added. The eigenvalues shift due to the overall change in $\epsilon_D$, and a greater shift is exhibited by eigenvalues of larger magnitude.}
\label{FFeigs_shift}
\end{figure}

\subsection{Stekloff Eigenvalues}

We now perform numerical tests for generalized Stekloff eigenvalues. In order to compute an approximate solution to the electric far field equation \eqref{FMeq}, we use the same matrix $A$ described for the computation of eigenvalues of the electric far field operator, and we add noise in the same manner. We first comment on the choice of the domain $B$ for both the unit cube and puck. The only requirement is that each scatterer is contained in $B$, but a natural choice is to choose $B$ to be a ball centered at the origin. We remark that when we solve the far field equation for each sampled value of $\lambda$, we do so for 10 randomly chosen $z$ in a ball (of radius 1/4 for the cube and 1/3 for the puck) contained inside $D$ and average the norms of the solutions to serve as our indicator function. In Figures \ref{Stekloff_D} and \ref{Stekloff} we plot the average norm of $g$, the solution obtained from applying Tikhonov regularization to \eqref{FMeq}, against the Stekloff parameter $\lambda$ for the cases in which $B=D$ and $B$ is a ball, respectively. We see that the peaks in the plot approximate the first couple of eigenvalues well for both the unit cube and the puck when $B$ is chosen to be a ball, but it is difficult to detect any eigenvalues reliably when $B=D$.

\begin{figure}
\begin{subfigure}{.5\textwidth}
\centering
\includegraphics[width=1\linewidth]{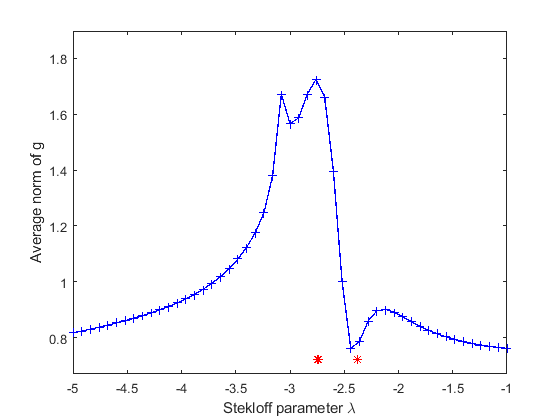}
\caption{unit cube, $B=D$}
\label{Stekloff_C}
\end{subfigure}
\begin{subfigure}{.5\textwidth}
\centering
\includegraphics[width=1\linewidth]{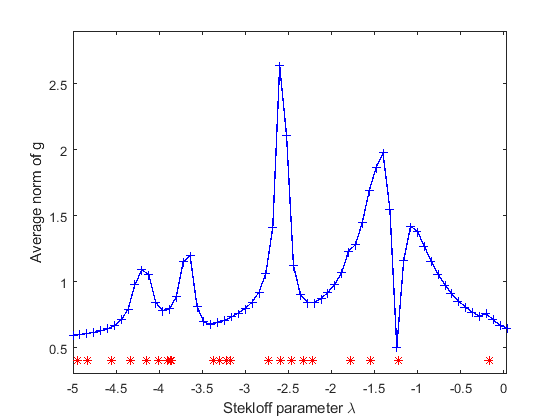}
\caption{puck, $B=D$}
\label{Stekloff_puck}
\end{subfigure}
\caption{A plot of the average norm of $g$ against the Stekloff parameter $\lambda$ with $\epsilon_D = 2.0$ and no noise, where $B=D$. The stars represent the exact eigenvalues computed using finite elements. We observe the difficulty in reliably detecting any eigenvalues.}
\label{Stekloff_D}
\end{figure}

\begin{figure}
\begin{subfigure}{.5\textwidth}
\centering
\includegraphics[width=1\linewidth]{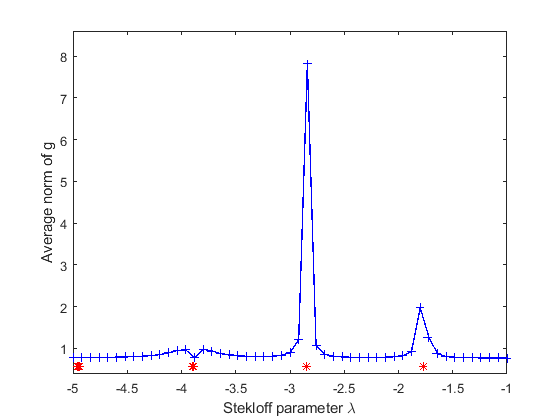}
\caption{unit cube, $B$ is the unit ball}
\label{Stekloff_CS}
\end{subfigure}
\begin{subfigure}{.5\textwidth}
\centering
\includegraphics[width=1\linewidth]{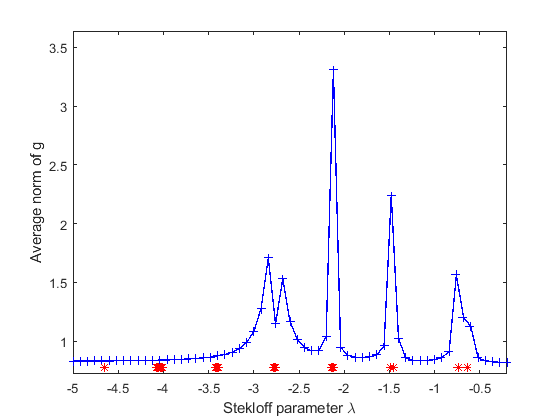}
\caption{puck, $B$ is the ball of a radius 1.7}
\label{Stekloff_puckS}
\end{subfigure}
\caption{A plot of the average norm of $g$ against the Stekloff parameter $\lambda$ with $\epsilon_D = 2.0$ and no noise, where $B$ is chosen to be a ball centered at the origin. The stars represent the exact eigenvalues computed using finite elements. We observe that the first couple of eigenvalues are detected in each case.}
\label{Stekloff}
\end{figure}

In Figures \ref{Stekloff_D_noise} and \ref{Stekloff_noise} we provide the same plots as in Figures \ref{Stekloff_D} and \ref{Stekloff}, respectively, for various levels of noise. For the case $B=D$, the plot for the cube exhibits a peak in the presence of noise which does not coincide with any of the eigenvalues, and a similar peak appears in the plot for the puck near the eigenvalue of smallest magnitude. For the case $B\neq D$, we observe that only a couple of the smallest eigenvalues in magnitude remain detectable in the presence of noise for both the unit cube and the puck, and the noise seems to reduce the prominence of the peaks rather than shift them.

\begin{figure}
\begin{subfigure}{.5\textwidth}
\centering
\includegraphics[width=1\linewidth]{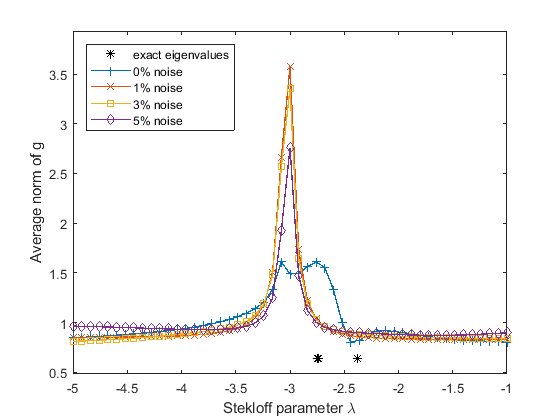}
\caption{unit cube, $B=D$}
\label{Stekloff_C_noise}
\end{subfigure}
\begin{subfigure}{.5\textwidth}
\centering
\includegraphics[width=1\linewidth]{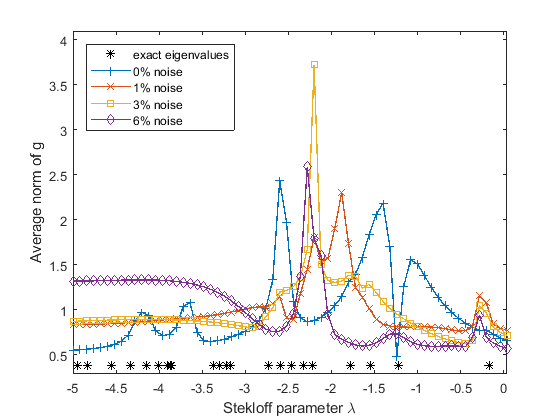}
\caption{puck, $B=D$}
\label{Stekloff_puck_noise}
\end{subfigure}
\caption{A plot of the average norm of $g$ against the Stekloff parameter $\lambda$ with $\epsilon_D = 2.0$ and $B=D$ for various levels of noise. The stars represent the exact eigenvalues computed using finite elements. Though some prominent peaks appear in the presence of noise for both scatterers, they do not correspond reliably to any of the eigenvalues.}
\label{Stekloff_D_noise}
\end{figure}

\begin{figure}
\begin{subfigure}{.5\textwidth}
\centering
\includegraphics[width=1\linewidth]{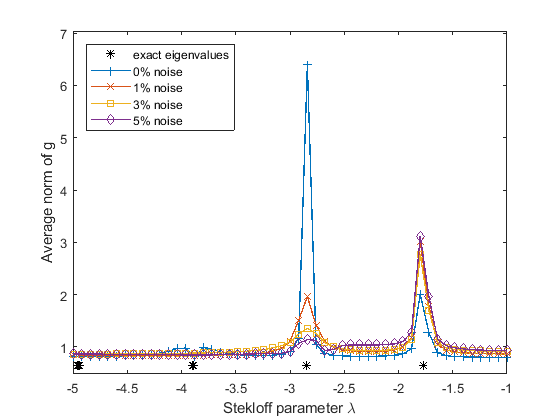}
\caption{unit cube, $B$ is the unit ball}
\label{Stekloff_CS_noise}
\end{subfigure}
\begin{subfigure}{.5\textwidth}
\centering
\includegraphics[width=1\linewidth]{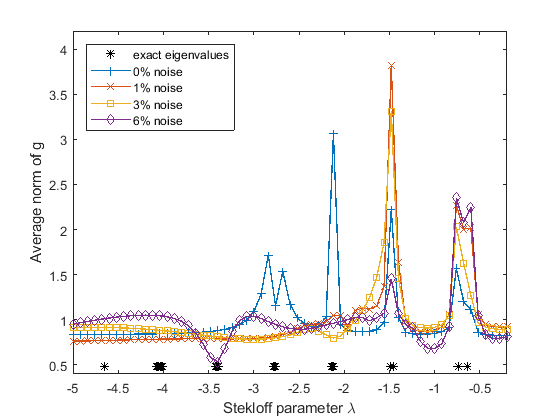}
\caption{puck, $B$ is the ball of a radius 1.7}
\label{Stekloff_puckS_noise}
\end{subfigure}
\caption{A plot of the average norm of $g$ against the Stekloff parameter $\lambda$ with $\epsilon_D = 2.0$ and $B\neq D$ for various levels of noise. The stars represent the exact eigenvalues computed using finite elements. Only a couple of eigenvalues remain detectable in the presence of noise.}
\label{Stekloff_noise}
\end{figure}

In Figures \ref{Stekloff_D_shift0} and \ref{Stekloff_shift0} we investigate the shift of generalized Stekloff eigenvalues due to an overall change in $\epsilon_D$ from 2 to 2.5. For the case $B=D$ we see that the exact eigenvalues shift and that there is some difference in the plot of the average norm of $g$, but since these two do not correspond well, it is difficult to make any definite conclusions about their usefulness in detecting changes in $\epsilon_D$. The case $B\neq D$ displays a reduced sensitivity in the eigenvalues, with only the smallest eigenvalues for the puck exhibiting any noticeable shift. However, this choice of $B$ improves the ability to detect eigenvalues and consequently this shift may be seen in the peaks of the plot of the average norm of $g$.

\begin{figure}
\begin{subfigure}{.5\textwidth}
\centering
\includegraphics[width=1\linewidth]{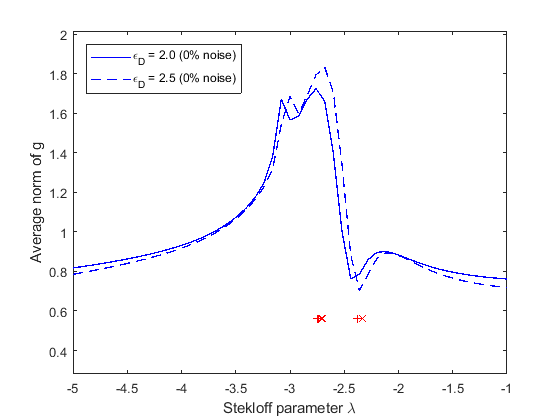}
\caption{unit cube, $B=D$}
\label{Stekloff_C_shift0}
\end{subfigure}
\begin{subfigure}{.5\textwidth}
\centering
\includegraphics[width=1\linewidth]{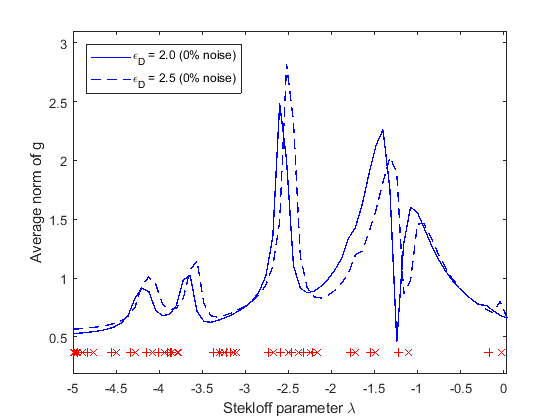}
\caption{puck, $B=D$}
\label{Stekloff_puck_shift0}
\end{subfigure}
\caption{A plot of the average norm of $g$ against the Stekloff parameter $\lambda$ with $\epsilon_D = 2.0, 2.5$ and no noise. The symbols `$+$' and `$\times$' represent the exact eigenvalues computed using finite elements for $\epsilon_D = 2.0$ and $\epsilon_D = 2.5$, respectively. The exact eigenvalues clearly shift and there is some difference in the plot of the indicator function due to the overall change in $\epsilon_D$.}
\label{Stekloff_D_shift0}
\end{figure}

\begin{figure}
\begin{subfigure}{.5\textwidth}
\centering
\includegraphics[width=1\linewidth]{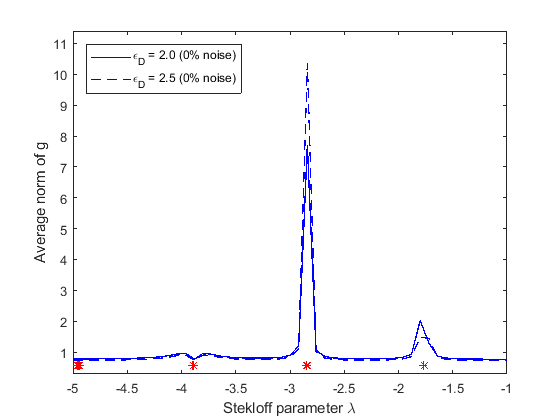}
\caption{unit cube, $B$ is the unit ball}
\label{Stekloff_CS_shift0}
\end{subfigure}
\begin{subfigure}{.5\textwidth}
\centering
\includegraphics[width=1\linewidth]{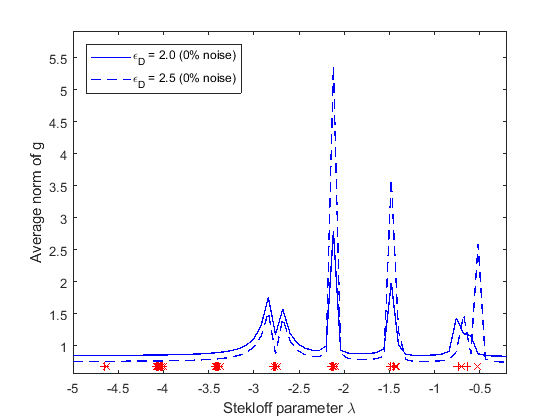}
\caption{puck, $B$ is the ball of a radius 1.7}
\label{Stekloff_puckS_shift0}
\end{subfigure}
\caption{A plot of the average norm of $g$ against the Stekloff parameter $\lambda$ with $\epsilon_D = 2.0, 2.5$ and no noise. The symbols `$+$' and `$\times$' represent the exact eigenvalues computed using finite elements for $\epsilon_D = 2.0$ and $\epsilon_D = 2.5$, respectively. We observe no noticeable shift in the eigenvalues for the unit cube, but we do observe a shift in the smallest eigenvalues for the puck.}
\label{Stekloff_shift0}
\end{figure}

The perturbation estimate \eqref{shift_estimate} suggests that the shift of a Stekloff eigenvalue due to a change in $\epsilon_D$ is related to the magnitude of a corresponding eigenfunction in a neighborhood of the change, and in Figures \ref{cube_plot} and \ref{puck_plot} we plot a cross section of an eigenfunction corresponding to the cube and puck, respectively. In Figure \ref{cube_eigenfunction} we see that $D$ is disjoint from the regions in which the eigenfunction $w$ is greatest, which suggests that an overall change in $\epsilon_D$ for the unit cube will not result in a large shift in the corresponding eigenvalue, as we observed. In contrast, we see in Figure \ref{puck_eigenfunction} that $D$ intersects with the regions of large magnitude of $w$ and explains the observed shift of the corresponding eigenvalue for the puck in Figure \ref{Stekloff_shift0}. Though precise knowledge of the geometry and material properties of the scatterer must be known in order to take advantage of this information, this relationship between the eigenfunctions and the material properties may be highly useful in nondestructive testing of materials. In particular, it might allow for the localization of flaws in a material by observing which eigenvalues shift and which do not.

\begin{figure}
\begin{subfigure}{.5\textwidth}
\centering
\includegraphics[width=.75\linewidth]{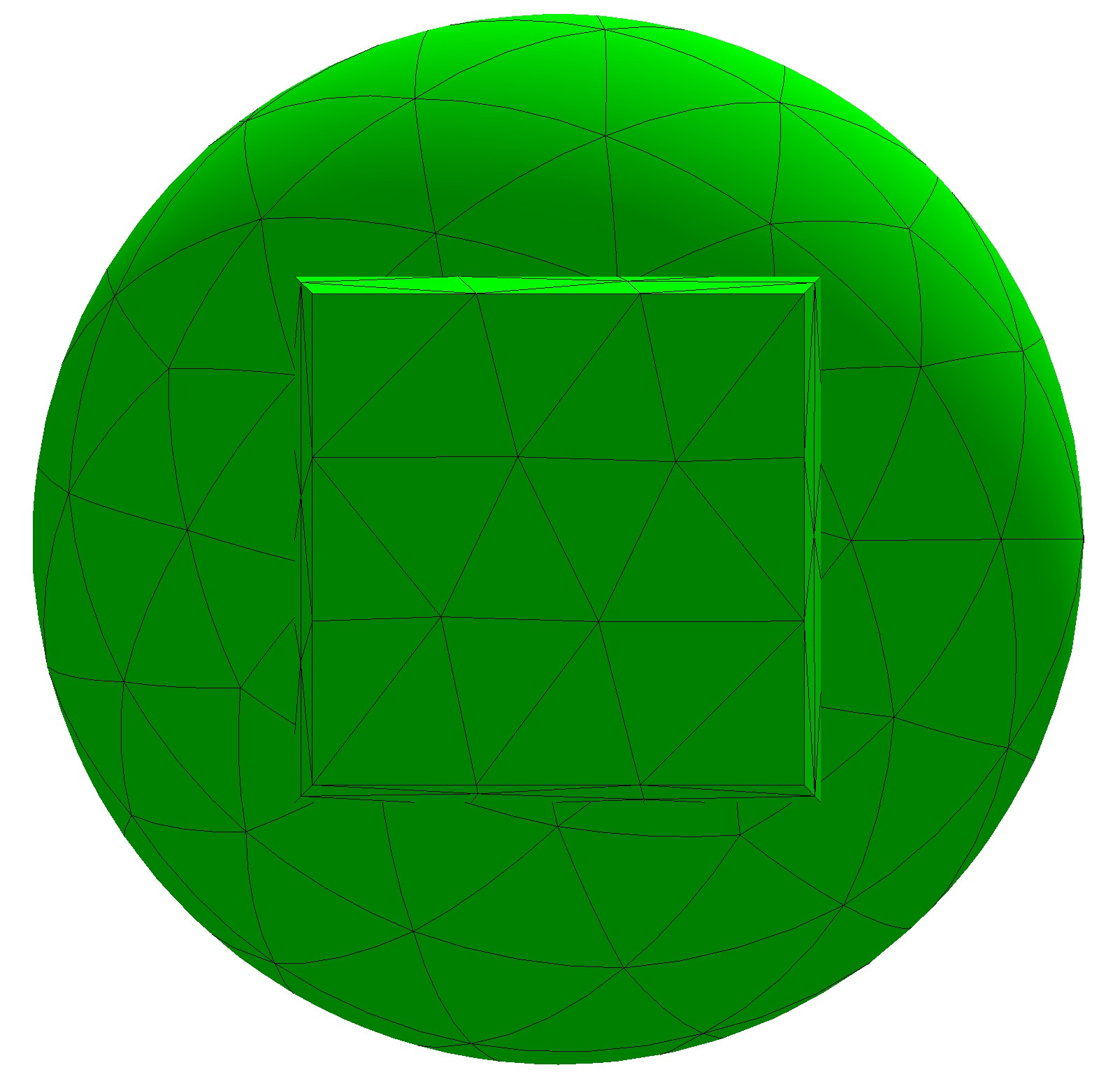}
\caption{Cross section of unit cube}
\label{cube_mesh}
\end{subfigure}
\begin{subfigure}{.5\textwidth}
\centering
\includegraphics[width=1\linewidth]{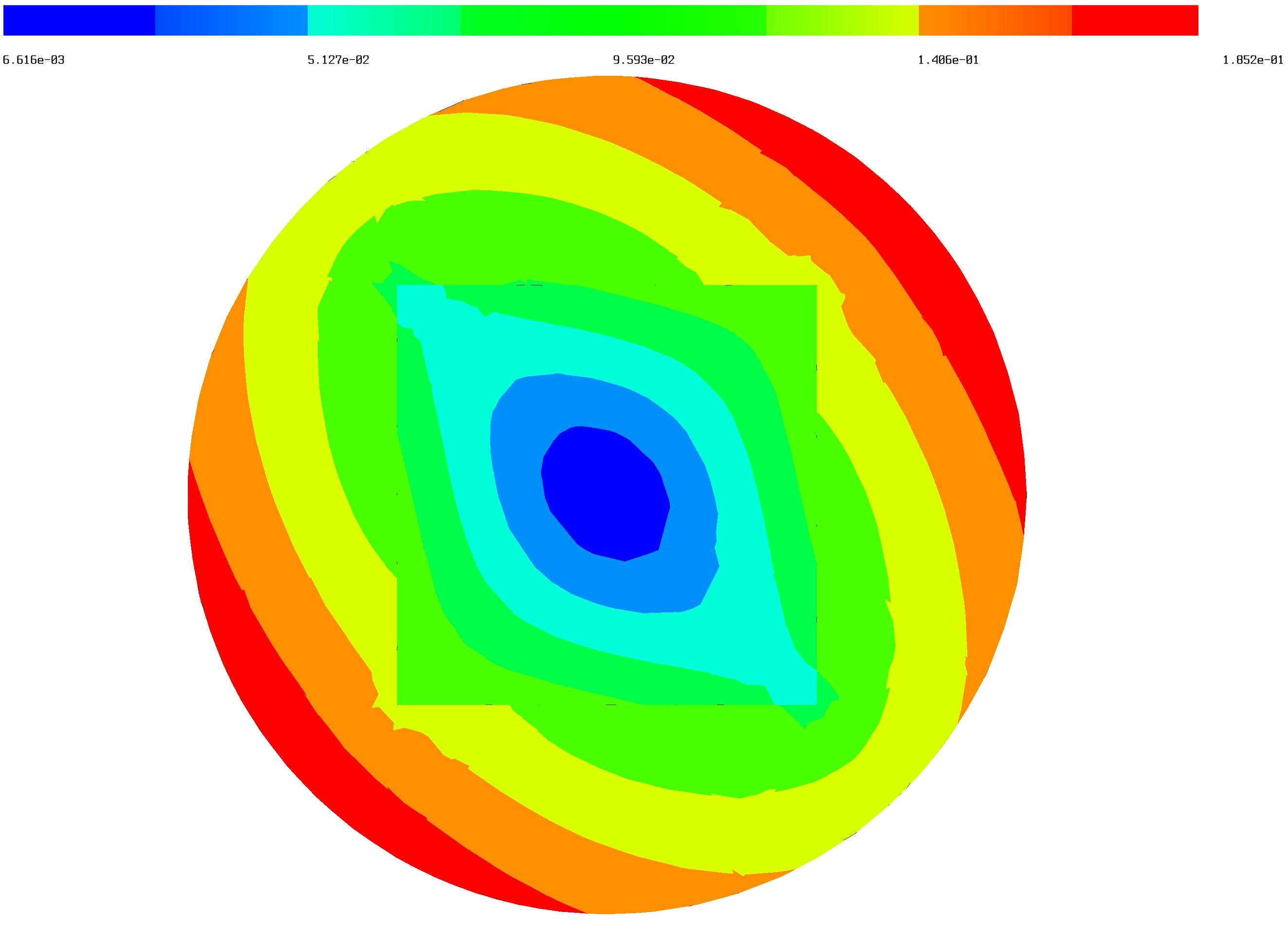}
\caption{Cross section of eigenfunction}
\label{cube_eigenfunction}
\end{subfigure}
\caption{A cross section of the unit cube surrounded by a ball and the corresponding cross section of an eigenfunction.}
\label{cube_plot}
\end{figure}

\begin{figure}
\begin{subfigure}{.5\textwidth}
\centering
\includegraphics[width=.75\linewidth]{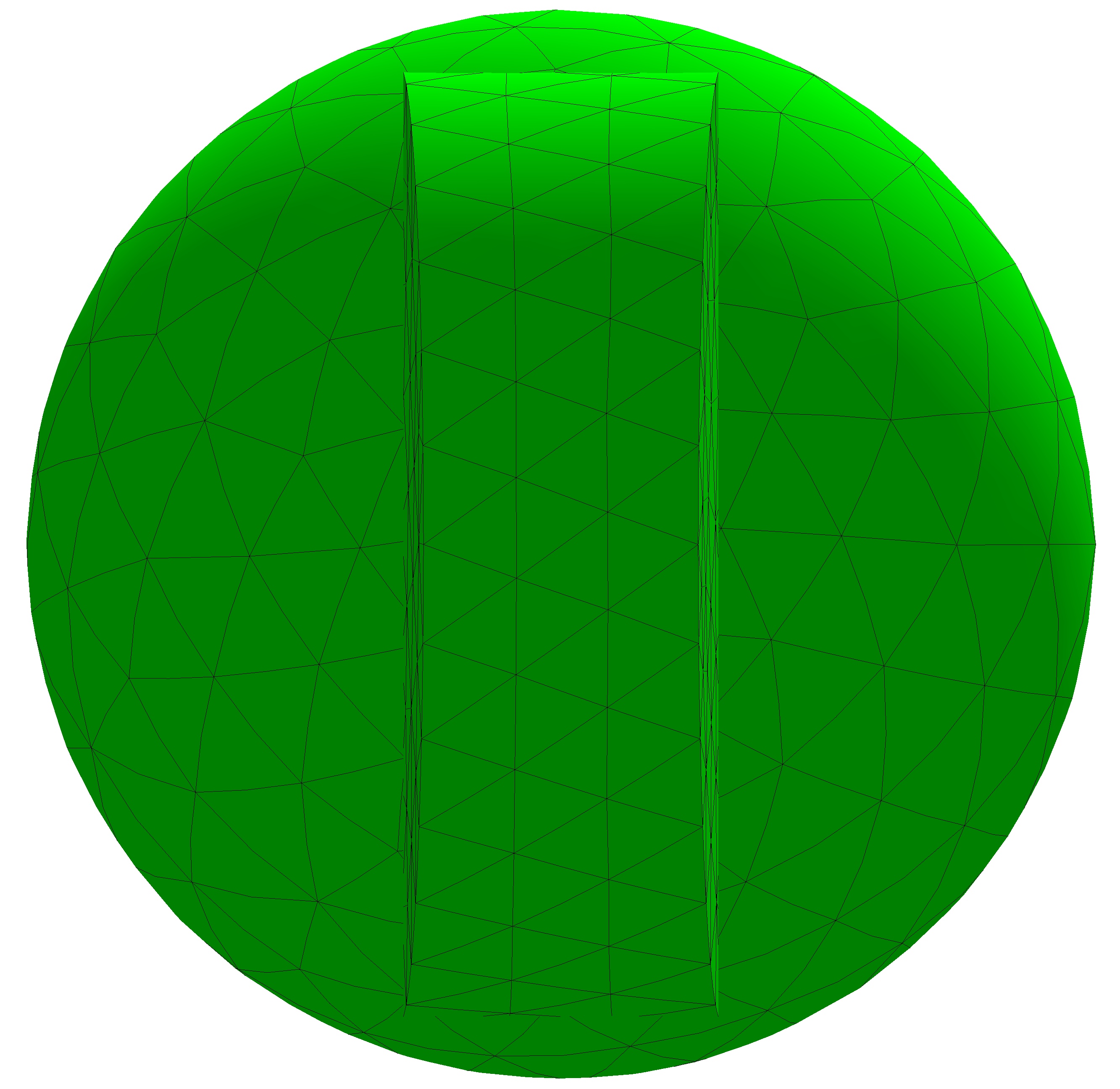}
\caption{Cross section of puck}
\label{puck_mesh}
\end{subfigure}
\begin{subfigure}{.5\textwidth}
\centering
\includegraphics[width=1\linewidth]{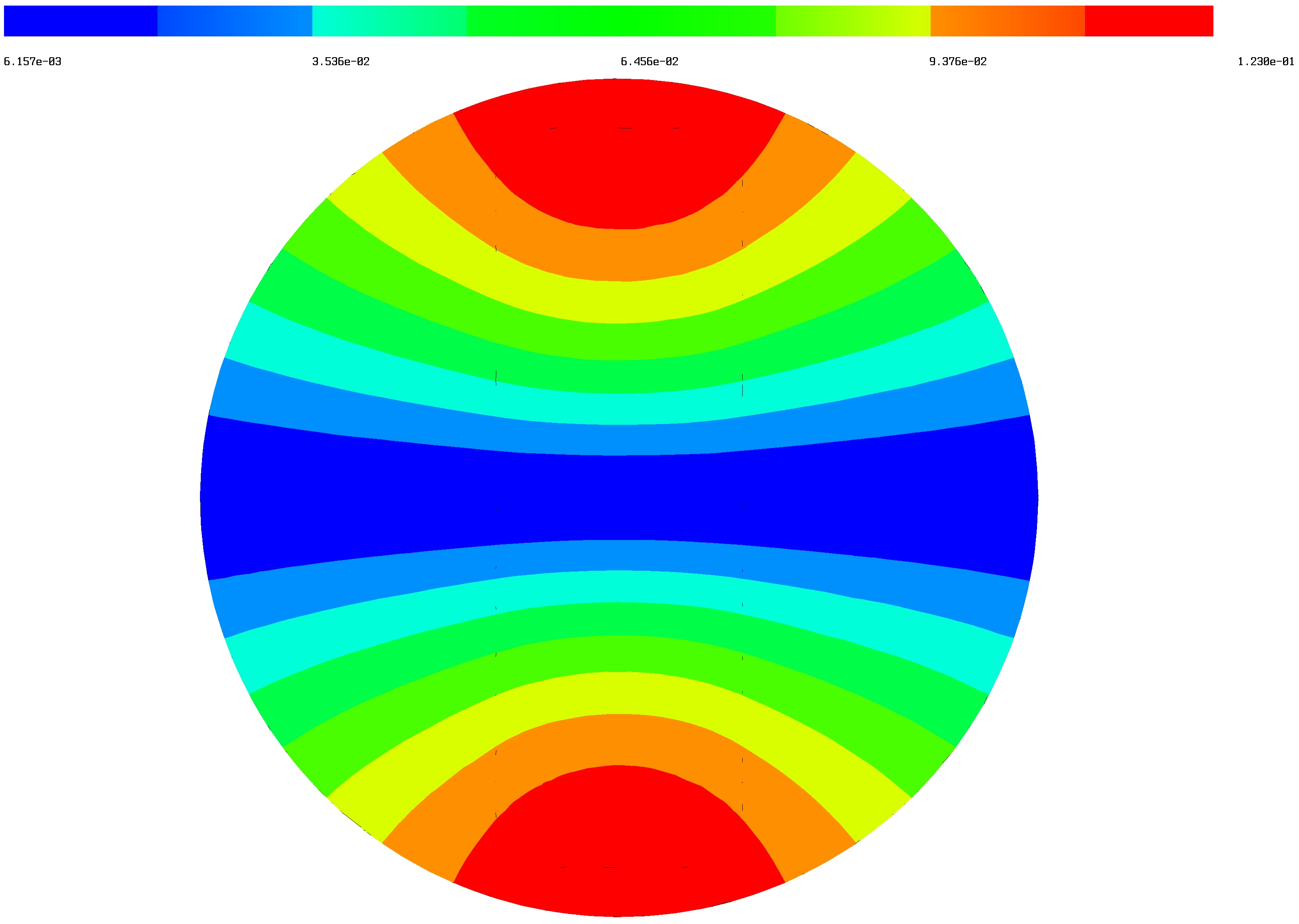}
\caption{Cross section of eigenfunction}
\label{puck_eigenfunction}
\end{subfigure}
\caption{A cross section of the puck surrounded by a ball and the corresponding cross section of an eigenfunction.}
\label{puck_plot}
\end{figure}

An important advantage of Stekloff eigenvalues over transmission eigenvalues is that Stekloff eigenvalues may in principle be computed for absorbing media, i.e. when $\epsilon_D$ has a nonzero imaginary part. Though the present theory does not include a proof of existence of electromagnetic Stekloff eigenvalues in this case, in Figures \ref{Stekloff_CS_complex} and \ref{Stekloff_puckS_complex} we present an example of their computation for the unit cube and the puck when $\epsilon_D = 2+2i$ and $B$ is chosen to be a ball. In these examples we have paired the plot for each scatterer with its noisy counterpart in order to obtain a more direct measure of the effect of noise. We observe that all of the eigenvalues in this sampling region are detected when no noise is present, and one remains detectable to a reasonable degree of accuracy in the presence of 7\% noise. It should be noted that the computational expense is greatly increased by the necessity to sample in a region of the complex plane rather than in an interval on the real line. However, as in the previous examples for real $\epsilon_D$, the computation of the modified Stekloff problems may be performed ahead of time for a given region $B$ and applied to any case in which $D\subseteq B$.

\begin{figure}
\begin{subfigure}{.5\textwidth}
\centering
\includegraphics[width=1\linewidth]{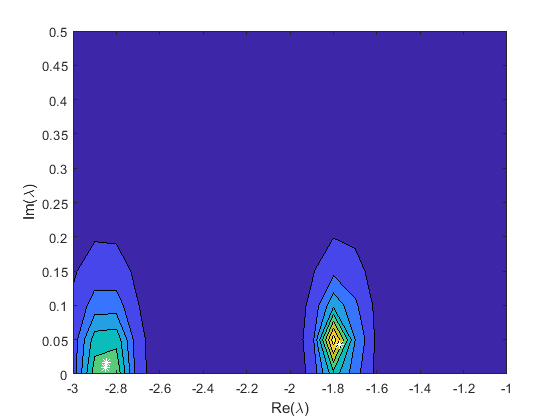}
\caption{unit cube, no noise}
\label{Stekloff_CS_complex_noise0}
\end{subfigure}
\begin{subfigure}{.5\textwidth}
\centering
\includegraphics[width=1\linewidth]{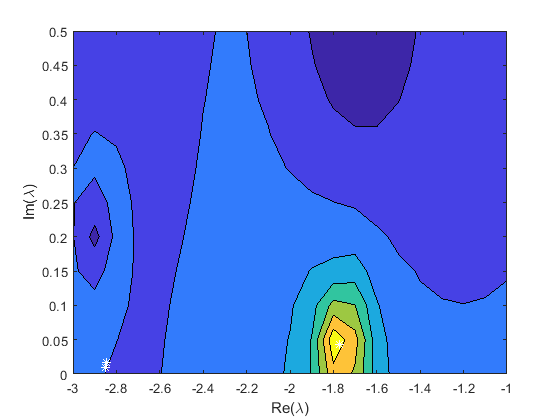}
\caption{unit cube, 7\% noise}
\label{Stekloff_CS_complex_noise7}
\end{subfigure}
\caption{A base 10 contour plot of the average norm of $g$ against the Stekloff parameter $\lambda$ in the complex plane for the unit cube with $\epsilon_D = 2+2i$ and two different noise levels. Here we choose $B$ to be the unit ball. The white stars represent the exact eigenvalues computed using finite elements. We observe that all of the eigenvalues in this region are detected when no noise is present, and one remains detectable with 7\% noise.}
\label{Stekloff_CS_complex}
\end{figure}

\begin{figure}
\begin{subfigure}{.5\textwidth}
\centering
\includegraphics[width=1\linewidth]{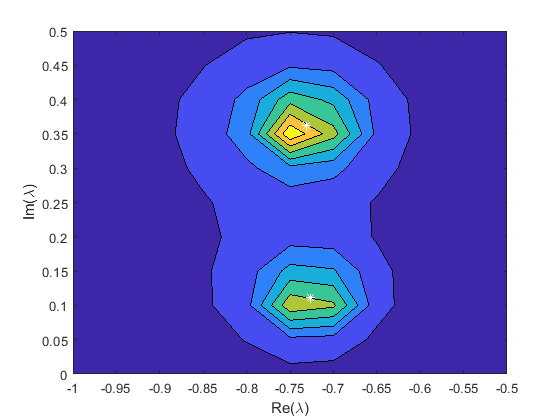}
\caption{puck, no noise}
\label{Stekloff_puckS_complex_noise0}
\end{subfigure}
\begin{subfigure}{.5\textwidth}
\centering
\includegraphics[width=1\linewidth]{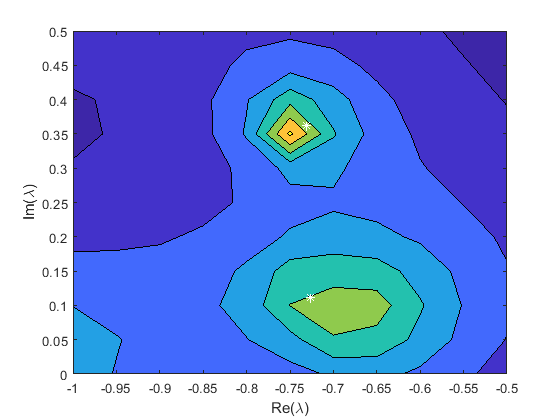}
\caption{puck, 7\% noise}
\label{Stekloff_puckS_complex_noise7}
\end{subfigure}
\caption{A base 10 contour plot of the average norm of $g$ against the Stekloff parameter $\lambda$ in the complex plane for the puck with $\epsilon_D = 2+2i$ and two different noise levels. Here we choose $B$ to be the unit ball. The white stars represent the exact eigenvalues computed using finite elements. We observe that all of the eigenvalues in this region are detected when no noise is present, and one remains detectable with 7\% noise.}
\label{Stekloff_puckS_complex}
\end{figure}

\section{Conclusion and Open Problems}

The fact that the electric far field data does not uniquely determine the material properties of an anisotropic medium presents many difficulties in the detection of changes in the material properties of a medium, and we have seen that various approaches using the idea of a target signature are available. An important question is which of these target signatures should be chosen for a given application, and unfortunately the answer is not entirely straightforward. Although the theory of transmission eigenvalues is applicable to dielectric media, the use of target signatures for absorbing media at this time rests with either the eigenvalues of the electric far field operator or generalized Stekloff eigenvalues, a pair with complementary strengths and weaknesses. On one hand, we have observed a noticeable shift in the eigenvalues of the electric far field operator due to an overall change in $\epsilon_D$, whereas Stekloff eigenvalues do not appear to shift as reliably. On the other hand, the relationship between Stekloff eigenvalues and the permittivity $\epsilon_D$ is apparent in the variational formulation and lends itself to investigation by standard techniques in the theory of partial differential equations, whereas little is known about the eigenvalues of the electric far field operator beyond their distribution in the complex plane. In addition, the use of Stekloff eigenvalues requires some decision-making on the choice of $B$: choosing $B=D$ often improves sensitivity at the expense of reliable detection of eigenvalues, and choosing $B\neq D$ improves the detection of eigenvalues while reducing their sensitivity to changes in the medium. Thus, any attempt to use these methods would require some experimentation to determine the best choice, and there are multiple trade-offs to consider. \par

However, the story likely does not end with this rather disappointing observation, as these are not the only target signatures under current study. In particular, there are a number of possible ways in which the electric far field operator can be modified. An example in acoustic scattering modifies the far field operator with that corresponding to scattering by an auxiliary homogeneous medium, and the eigenparameter of interest $\eta$ is the index of refraction of the auxiliary medium \cite{audibert_cakoni_haddar,cogar_colton_meng_monk}. An important advantage of this method is that the auxiliary scattering problem also depends on an additional parameter $\gamma$ which may be tuned to improve the sensitivity of the eigenvalues to changes in the material properties, thus overcoming the loss of sensitivity resulting from the choice $B\neq D$. \par

In Figure \ref{stekloff_mte} we show a direct comparison between Stekloff eigenvalues and these so-called modified transmission eigenvalues for acoustic scattering of an L-shaped domain, where we have used the recently developed generalized linear sampling method (cf. \cite{audibert_cakoni_haddar}) in order to detect the eigenvalues from far field data. This domain has been used for numerical testing of Stekloff eigenvalues and modified transmission eigenvalues previously (cf. \cite{stek1} and \cite{cogar_colton_meng_monk}, respectively), and we see that the shift in the eigenvalues due to a circular flaw located at $(x_c,y_c) = (0.1,0.4)$ of radius $r_c = 0.05$ is much more pronounced for modified transmission eigenvalues than Stekloff eigenvalues. It should be noted that for the case of Stekloff eigenvalues there exist peaks in the GLSM indicator corresponding to some of the other exact eigenvalues shown, but the height of these peaks is considerably less than the one visible. We remark that the modified transmission eigenvalues correspond to the choice $\gamma = 0.5$ in \cite{cogar_colton_meng_monk} and that instead using $\gamma = 2$ produces poor results. \par

This example indicates that, at least for acoustic scattering and with a proper choice of $\gamma$, modified transmission eigenvalues provide more information about the material properties of the scatterer than Stekloff eigenvalues. This observation is not too surprising, as can be seen from the fact that for spherically stratified media there exists a single Stekloff eigenvalue corresponding to a spherically symmetric eigenfunction, whereas there exist infinitely many such modified transmission eigenvalues. Extending this approach to Maxwell's equations is the focus of our current research.

\begin{figure}
\begin{subfigure}{.5\textwidth}
\centering
\includegraphics[width=1\linewidth]{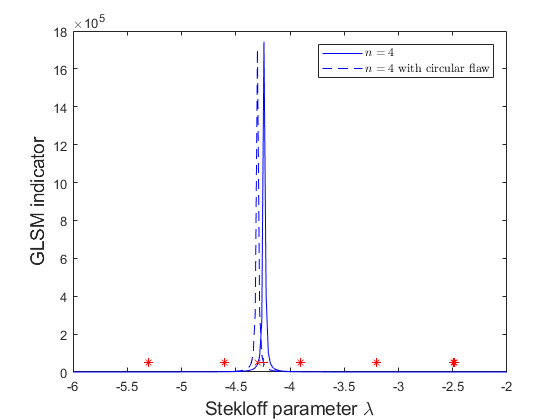}
\caption{Stekloff eigenvalues}
\label{stekloff_Lshape}
\end{subfigure}
\begin{subfigure}{.5\textwidth}
\centering
\includegraphics[width=1\linewidth]{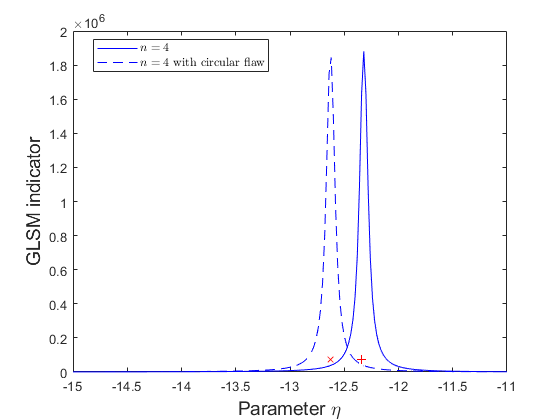}
\caption{Modified transmission eigenvalues}
\label{mte_Lshape}
\end{subfigure}
\caption{A direct comparison of Stekloff eigenvalues and modified transmission eigenvalues (with $\gamma = 0.5$) for acoustic scattering by an L-shaped domain. The shift in the eigenvalues due to a circular flaw located at $(x_c,y_c) = (0.1,0.4)$ of radius $r_c = 0.05$ is much more pronounced for modified transmission eigenvalues than Stekloff eigenvalues. The red '+' symbol represents the exact eigenvalues for the unflawed domain, and the red '$\times$' symbol represents the exact eigenvalues for the domain with a circular flaw.}
\label{stekloff_mte}
\end{figure}

\FloatBarrier

\section*{Acknowledgements}

This material is based upon work supported by the Army Research Office through the National Defense Science and Engineering Graduate (NDSEG) Fellowship, 32 CFR 168a, and by the Air Force Office of Scientific Research under award number FA9550-17-1-0147.

\bibliographystyle{myplain}
\bibliography{survey_references}

\end{document}